\documentclass[12pt]{amsart}

\numberwithin{equation}{section}

\usepackage{amsmath,amsthm,amsfonts,amscd,eucal}
\usepackage{url}
\usepackage{graphicx}
\usepackage{subfig}
\usepackage{tikz}
\usetikzlibrary{arrows.meta}
\usepackage{amssymb}
\def\ca{{\mathcal A}}
\def\cb{{\mathcal B}}

\def\cf{{\mathcal F}}
\def\cg{{\mathcal G}}
\def\ch{{\mathcal H}}

\def\ck{{\mathcal K}}

\def\cam{{\mathcal M}}

\def\car{{\mathcal R}}

\def\a{\alpha}

\def\g{\gamma}  
\def\d{\delta}  

\def\eps{\varepsilon}

\def\l{\lambda} 

\def\m{\mu}

\def\n{\nu}
\def\r{\rho}
\def\s{\sigma}

\def\om{\omega} \def\Om{\Omega}

\def\gf{{\mathfrak F}}
\def\gpg{{\mathfrak g}}

\def\bc{{\mathbb C}}

\def\bn{{\mathbb N}}

\def\br{{\mathbb R}}

\def\bz{{\mathbb Z}}

\def\bfu{{\mathbf u}}

\newtheorem{thm}{Theorem}[section]
\newtheorem{lem}[thm]{Lemma}

\newtheorem{prop}[thm]{Proposition}

\theoremstyle{definition}
\newtheorem{rem}[thm]{Remark}

\def\dim{\mathop{\rm dim}}
\def\min{\mathop{\rm min}}
\def\supp{\mathop{\rm supp}}
\def\Im{\mathop{\rm Im}}

\def\sp{\mathop{\rm sp}}

\def\supp{\mathop{\rm supp}}

\def\di{\mathop{\rm d}\!}
\def\di{{\rm d}}

\def\sp{\mathop{\rm span}}

\definecolor{ao(english)}{rgb}{0.0, 0.5, 0.0}
\definecolor{auburn}{rgb}{0.43, 0.21, 0.1}
\definecolor{aquamarine}{rgb}{0.5, 1.0, 0.83}

\begin{document}

\title[distributions for nonsymmetric monotone operators]{distributions for nonsymmetric monotone and weakly monotone
position operators}
\author{Vitonofrio Crismale}
\address{Vitonofrio Crismale\\
Dipartimento di Matematica\\
Universit\`{a} degli studi di Bari\\
Via E. Orabona, 4, 70125 Bari, Italy}
\email{\texttt{vitonofrio.crismale@uniba.it}}
\author{Maria Elena Griseta}
\address{Maria Elena Griseta\\
Dipartimento di Matematica\\
Universit\`{a} degli studi di Bari\\
Via E. Orabona, 4, 70125 Bari, Italy}
\email{\texttt{maria.griseta@uniba.it}}
\author{Janusz Wysocza\'nski}
\address{Janusz Wysocza\'nski\\
Institute of Mathematics\\
Wroclaw University\\
pl. Grunwaldzki 2/4, 50-384 Wroclaw, Poland}
\email{{\texttt{jwys@math.uni.wroc.pl}}}

\date{}

\begin{abstract} We study the vacuum distribution, under an appropriate scaling, of a family of partial sums of nonsymmetric position operators on weakly monotone and monotone Fock spaces, respectively. We preliminary treat the case of weakly monotone Fock space, and show that any single operator has the vacuum law belonging to the free Meixner class. After establishing some relations between the combinatorics of Motzkin and Riordan paths, we give a recursive formula for the vacuum moments of the law of any finite sum. Since the operators are monotone independent, the distribution is the monotone convolution of the free Meixner law above. We also investigate the asymptotic measure for these sums, which can be seen as "Poisson type" limit law. It turns out to belong to the free Meixner class, with an atomic and an absolutely continuous part (w.r.t. the Lebesgue measure). Finally, we briefly apply analogous considerations to the case of monotone Fock space.

\end{abstract}

\maketitle

\section{introduction}
\label{intro}
The present notes are a continuation of the investigation started by the authors in \cite{CGW}. There we studied the vacuum distribution of sums of symmetric position operators on the weakly monotone Fock space (WM-Fock for short), based on a separable Hilbert space $\ch$. The structure of WM-Fock was first defined in \cite{JWys2005}, where the interested reader is referred for a general presentation. A symmetric position operator is obtained by adding the basic operators, namely the creators $A_i^{\dagger}$ and annihilators $A_i$ by the \textit{test functions} given by unit vectors $e_i$ on $\ch$ (details are explained in Section 2), and in noncommutative (free, monotone or boolean) probability it is often called \textit{gaussian}. One of the reasons for such a name can be traced back to the notions of stochastic independence in the noncommutative realm. Indeed, when one drops commutativity, several inequivalent notions of independence appear among the random variables \cite{S,Mur3}, and consequently the appropriate central limit theorems give rise to limit laws different from the normal distribution. The operators above are therefore called gaussian since they realize, in suitable Hilbert spaces, natural examples of noncommutative random variables whose distributions, w.r.t. a distinguished vector state, are central limit laws. In fact, it was observed by Hudson and Parthasarathy \cite{HP} that the classical Brownian motion can be realized on the \textit{symmetric Fock space} as the sum of creation and annihilation processes. In particular, the distribution of such operators w.r.t. the vacuum state is the gaussian law. This idea turned out to work also in free probability, where the sum of creation and annihilation operators acting on the \emph{free (full) Fock space}, is the Wigner semicircle law, which is the central limit distribution for freely independent random variables. Similarly, Muraki \cite{Mur1997} has constructed Brownian motion with monotone independent increments on the (\emph{continuous} type) \emph{monotone Fock space} as the sum of creation and annihilation processes, where every random variable in the family has the arcsine law. The latter is the central limit distribution in monotone probability, and a generalization of Muraki's construction has been obtained in \cite{KW} for Brownian motions with multidimensional time parameter taken from a positive symmetric cone and with bm-independent increments.

On the other hand, the sum of creation and annihilation operators on the \emph{discrete monotone Fock} space has just the Bernoulli distribution, which is far away from the arcsine law (see, {\it e.g.} \cite{CrLu2} and the references therein). The construction of the WM-Fock space was the attempt to build a discrete model in which the position operators would have the arcsine distribution, and the result was partially satisfactory, since only some finite rank perturbations of the position operators were proved to have the arcsine distribution. In \cite{CGW} we showed that a single position operator on the WM-Fock space has the Wigner semicircle distribution, which is also away from the arcsine law, but still absolutely continuous. We also noticed that sums of such position operators have absolutely continuous distributions and, by monotone independence, the latter are indeed the monotone convolution powers of the Wigner semicircle law.

The problem of finding suitable noncommutative random variables whose vacuum distributions of their sums weakly converge, under appropriate scaling, to the analogue of Poisson measure in the classical case (law of small numbers), has been extensively studied in noncommutative probability. The operators involved are generally constructed on appropriate Fock spaces as the combination $A^{\dagger}+A+\lambda A^{0}$ of a position operator $A^{\dagger}+A$ and a scalar multiple $\lambda > 0$, called the \textit{intensity}, of the so-called preservation operator $A^0$. They are said nonsymmetric position operators, and usually one defines $A^0:=A^{\dagger}A$, the first example in the literature going back to \cite{AB}. This is the way we do it as well, following the free \cite{Sp} and the monotone \cite{Mur1999} probability cases. The intensity $\lambda$ introduces additional structure into the considerations, and is therefore worthwhile to study. This structure usually is related to the combinatorics of partitions combined with block orderings or labelings. Moreover, in the weakly monotone case, it significantly enlarges the class of partitions compared to the position operator case from \cite{CGW}, in which (for central limits) only noncrossing pair partitions appear. In free probability the analogue of the classical Poisson law is the Marchenko-Pastur distribution (see, {\it e.g.} \cite{Sp}), whose moments can be computed by taking all noncrossing partitions and putting the parameter (label) $\lambda$ on each block. A generalization of this result has been obtained in \cite{BLS} for conditionally free convolution.

The situation in monotone probability is not as clear. Muraki found only implicit form of the distribution for the Poisson type limit theorem for the monotone convolution, however its moments are related to all noncrossing partitions with monotone order on blocks (see \cite{Mur,HaS}).

In this paper we study the Poisson type limit for random variables $X_i(\l):=G_i+\l A_i^{0}$.
Here, $\l>0$, $G_i:=A_i+A_i^{\dag}$, and $A_i$, $A_i^{\dag}$, $A^{0}_i:=A^{\dag}_iA_i$ denote respectively the annihilation, creation and preservation operators, by orthonormal basis vectors $e_i\in \ch$ on the WM-Fock space $\gf_{WM}(\ch)$. In particular, after denoting
$$
S_m(\l):=\frac{G_1+\ldots+G_m}{\sqrt{m}}+ \l(A_1^{0}+\ldots+A_m^{0})\,, \quad m\in\bn\,,\,\, \l>0\,,
$$
we compute
$$
\lim_{m\to\infty}\langle S_m(\l)^n\Om,\Om\rangle
$$
for each $n\in\bn$, where $\Om$ is the so-called weakly monotone vacuum vector. For $n$ running in the positive integers, the limits above describe a sequence of moments of a distribution $\n_\l$. In the final part of these notes we compute $\n_\l$, and show that it belongs to the free Meixner class, and it is the same limit measure as the one for analogous sums of nonsymmetric monotone position operators, which first appeared in \cite{Mur1999}.

The paper is organized as follows. In Section \ref{prel} we present the WM-Fock space with creation, annihilation and preservation operators, and recall basic information about labeled noncrossing partitions, combinatorics of Motzkin and Riordan paths, basic properties of the Cauchy transform of a probability measure, and free Meixner laws as well. Based on Flajolet's results \cite{F1}, in Section \ref{moments} the vacuum law of any nonsymmetric position operator is seen to belong to the free Meixner class in Theorem \ref{measacRio}. More precisely, the law is absolutely continuous w.r.t. the Lebesgue measure when $0<\l\leq1$, and has both atomic and absolutely continuous parts if $\l>1$. In addition, we highlight some relations between moments and Motzkin and Riordan paths. As nonsymmetric position operators with intensity $\l$ are monotone independent random variables in the $*$-algebra generated by weakly monotone creation operators (see \cite{CGW}), one naturally tries to get some information on the vacuum law of their partial sums $\sum_{k=1}^m (G_k+\l A^0_k)$, which correspond to the $m$-fold monotone convolution of the free Meixner law aforementioned. To this aim is devoted Subsection \ref{sec3.1}, where in Theorem \ref{mommgeq2} we find a recursive formula for the vacuum moments of $T_m(\l_m):=\sqrt{m}S_m(\l)$ via \textit{noncrossing weakly monotone labeled partitions}, and moreover an explicit computation of the Cauchy transform of the measure in the case $m=2$.
In Section \ref{sec5} we obtain the limit distribution of $S_m(\l)$ in Theorem \ref{clt} as a free Meixner law containing atomic and absolutely continuous part. As one could expect it reduces to the standard arcsine law when $\l$ equals zero. Finally, in Section \ref{secfin}, we perform analogous considerations for the case of monotone Fock space. Although the vacuum law of a single nonsymmetric position operator here is a two-points discrete measure, the Poisson type limit is the same as in the weakly monotone case.

\section{preliminaries}
\label{prel}
In this section we give a miscellany of definitions, notations and some known results frequently used in the sequel.
\subsection{Weakly Monotone Fock space}
\label{weakmon}

Let $\ch$ be a separable Hilbert space with a fixed orthonormal basis $(e_i)_{i\geq 1}$. By $\gf(\ch)$ we denote the full Fock space on $\ch$, whose vacuum vector is $\Om=1\oplus0\oplus\cdots$. The \emph{weakly monotone Fock space}, in the sequel denoted by $\gf_{WM}(\ch)$, is the closed subspace of $\gf(\ch)$ spanned by $\Om$, $\ch$ and all the simple tensors of the form $e_{i_k}\otimes e_{i_{k-1}}\otimes \cdots\otimes e_{i_1}$, where $i_k\geq i_{k-1}\geq\cdots\geq i_1$
and $k\geq2$.

If the Hilbert space $\ch$ is finite dimensional with $n =\dim(\ch)\geq 2$, then the basis for $\gf_{WM}(\ch)$ consists of the vacuum and all the simple tensors
\begin{equation}
\label{basis}
e_n^{k_n}\otimes e_{n-1}^{k_{n-1}}\otimes\cdots\otimes e_1^{k_1}\,,
\end{equation}
where $k_n, k_{n-1},\ldots,k_1\geq 0$, $e_m^{k}:=\underbrace{e_m\otimes\cdots\otimes e_m}_k$ if $k\geq1$, and the convention that $e^{k_i}_{i}$ does not appear in \eqref{basis} if $k_i=0$.

The weakly monotone creation and annihilation operators with "test function" $e_i$, denoted by $A^\dag_i$ and $A_i$ respectively, are defined on the linear generators as follows. For any $i_k\geq i_{k-1}\geq \cdots \geq i_1$, $k\geq 2$, and $j\geq 1$
\begin{align*}
&A_i\Om=0\,, \quad A_ie_j=\delta_{i,j}\Om\,, \\
& A_i(e_{i_k}\otimes e_{i_{k-1}}\otimes \cdots\otimes e_{i_1})=\delta_{i,i_k}e_{i_{k-1}}\otimes \cdots\otimes e_{i_1}\,,
\end{align*}
where $\delta_{i,j}$ is the Kronecker symbol, and
\begin{align*}
&A^{\dag}_i \Om=e_i\,, \quad A^{\dag}_ie_j=\a_{i,j}e_i\otimes e_j\,, \\
&A^{\dag}_i(e_{i_k}\otimes e_{i_{k-1}}\otimes \cdots\otimes e_{i_1})=\a_{i,i_k} e_i\otimes e_{i_k}\otimes e_{i_{k-1}}\otimes \cdots\otimes e_{i_1}\,,
\end{align*}
where
\begin{equation*}
\a_{j,k}=\begin{cases}
1 & \text{if $j\geq k$,}\\
0 & \text{otherwise.}
\end{cases}
\end{equation*}
They can be extended by linearity and continuity to the whole $\gf_{WM}(\ch)$, are adjoint to each other, and of norm one. Furthermore, they satisfy the following relations
\begin{equation}
\label{cr}
\begin{split}
A^{\dag}_iA^{\dag}_j&=A_jA_i =0 \quad\text{if $i<j$,}\\
A_iA^{\dag}_j& =0 \,\,\,\,\,\,\,\,\,\,\,\,\,\,\,\,\,\,\,\,\,\,\,\,\,\,\, \text{if $i\neq j$.} \\
\end{split}
\end{equation}
As it will be useful in the sequel, we report here Lemma 2.1 from \cite{CGW} for the convenience of the reader.
\begin{lem}
\label{lem2.1WM}
For any $k,j\geq1$, one has
\begin{align}
\label{alpha}
A_kA_jA_j^{\dag}&=\a_{j,k}A_k  & A_jA_j^{\dag}A_k^{\dag}=\a_{j,k}A_k^{\dag}\,.
\end{align}
Moreover, for $j\geq k$
\begin{align}
\label{jgeqk}
A_jA^{\dag}_jA_k&=A_k  & A_k^{\dag}A_jA_j^{\dag}=A_k^{\dag}\,.
\end{align}
\end{lem}
For each $i\in\bn$, we denote $A^0_i:= A^{\dag}_iA_i$ the orthogonal projection onto
$$
\gf_{WM}^{= i}(\ch):=\overline{\sp\{e_{k_m}\otimes \cdots \otimes e_{k_1}: i=k_m\geq \ldots \geq k_1, \ m\geq 1\}}\,.
$$
As usual (see, \emph{e.g.} \cite{AB}) we call it preservation (or conservation) operator.
Its relations with the other basic operators in $\gf_{WM}(\ch)$ are summarised below, and follow from \eqref{cr}
\begin{equation}
\label{crcons}
\begin{split}
A^0_iA^{\dag}_j&=A_jA^0_i =0 \quad \text{if $i\neq j$}\\
A^0_iA^0_j &=0 \,\,\,\,\,\,\,\,\,\,\,\,\,\,\,\,\,\,\,\,\,\,\,\,\,\,\,\text{if $i\neq j$}\\
A^0_iA_j &=A^{\dag}_jA^0_i =0 \quad\text{if $i> j$}.\\
\end{split}
\end{equation}

\subsection{Partitions of a finite set}
\label{partitions}

For $n,m\in\bn$ with $n<m$, let us denote $[n,m]:=\{n,n+1,\ldots, m\}$, and $[n]:=\{1,\ldots,n\}$. Let $S$ be a nonempty linearly ordered finite set. The collection $\pi=\{B_1,\ldots, B_p\}$ is a \emph{partition} of the set $S$ if the $B_i$ are disjoint nonempty subsets whose union is $S$. The $B_i$ are also called blocks of the partition $\pi$. Any $B_k\in\pi$ is called a \emph{singleton} if $|B_k|=1$, where $|\cdot|$ denotes the cardinality. The number of blocks of $\pi$ is denoted by $|\pi|$, and the set of all the partitions of $S$ is $P(S)$. In the special case $S=[n]$ we denote this set by $P(n)$. A similar notation for subsets of $P(n)$ will be introduced in the next paragraph.

A block $B$ of $\pi\in P(n)$ is called an \emph{interval block} if $B=[l,m]$ for some $1\leq l\leq m\leq n$. An \emph{interval partition} of $[n]$ is an element of $P(n)$ for which every block is an interval. In addition, the set $V(n)$ of interval partitions on $[n]$ is in bijection to the \emph{compositions} of $n$, \emph{i.e} the sequences of integers $(q_1,\ldots,q_j)$ such that $q_h>0$, and $q_1+\cdots + q_j=n$.

A partition $\pi$ has a \emph{crossing} if it contains at least two distinct blocks $B_i$ and $B_j$, and elements $v_1,v_2\in B_i$, $w_1,w_2\in B_j$ s.t. $v_1<w_1<v_2<w_2$. Otherwise, it is called \emph{noncrossing}. We denote by $NC_{2,1}(S)$ the set of noncrossing partitions such that each block contains at most two elements. If $|S|$ is even, $NC_2(S)$ denotes the set of noncrossing partitions with all blocks of cardinality $2$.

A noncrossing partition of $[n]$ is called \emph{irreducible} if there is a block containing both $1$ and $n$. Note that every partition is uniquely decomposed into irreducible (sub)partitions, generally called factors \cite{AHLV}.

A block $B_k\in\pi$ is called \emph{inner} if there exist $B_h\in\pi$ and $v_1,v_2\in B_h$ such that $v_1<w<v_2$ for all $w\in B_k$. A block is called \emph{outer} if it is not inner, and $\pi$ is noncrossing. In this case, the subset collecting the partitions without outer singletons is denoted by $PI(S)$. From now on, we put
$$
NCI_{2,1}(n):=NC_{2,1}(n)\cap PI(n).
$$

For $\pi$ a noncrossing partition of $[n]$ with blocks $\pi=\{B_1, \ldots , B_k\}$, one has a natural partial  order $\preceq_{\pi}$ on the blocks given by $B_i\preceq_{\pi}B_j$ if $B_j$ is nested inside $B_i$, i.e. $\min B_i\leq \min B_j \leq \max B_j \leq \max B_i$, where $\min B$ (resp. $\max B$) denotes the minimal (resp. maximal) element of the block $B$. When the extremal inequalities above are strong, we write $B_i\prec_{\pi}B_j$.

Let $\pi=\{B_1, \ldots , B_k\}$ be a partition of $[n]$ and $J$ be a set. A \emph{labeling} with values in $J$ is a function $L: \pi \rightarrow J$, and the pair $(\pi,L)$ is called a \emph{labelled partition}. We denote $LP(J,n)$ the set of labeled partitions with labels in $J$.

In what follows the aforementioned partition $\pi$ has to be taken noncrossing. Recall that for $x,y$ arbitrary elements of a poset $(S,\leq)$, one says that $y$ \emph{covers} $x$ if $x<y$ and there is no $z$ such that $x<z<y$. If $\pi=\{B_1,\ldots, B_k\}\in LP(J,n)$ with label function $L$ is such that $L(B_j)=L(B_i)$ for every singleton $B_j$ which covers a block $B_i$ w.r.t. $\preceq_{\pi}$, we say that $\pi$ belongs to $NCLP_s(J,n)$.

The label function $L$ is called \emph{weakly monotone} if it preserves the ordering of the blocks, i.e. $L(B_i)\leq L(B_j)$ when $B_i\preceq_{\pi} B_j$. A weakly monotone label function is called \emph{monotone} if $L(B_i)< L(B_j)$ when $B_i$ and $B_j$ both contain at least two elements, and $\min B_i< \min B_j < \max B_j < \max B_i$. The labelled partitions belonging to $NCLP_s(J,n)$ whose label function is weakly monotone (monotone) are denoted by $NCLP_s^{wm}(J,n)$ ($NCLP_s^{m}(J,n)$).
Finally, by $ANC^{wm}(J,n)$ we denote the noncrossing weakly monotone partitions on $[n]$ without outer singletons, with labels in $J$ and blocks of cardinality at most 2, and finally belonging to $NCLP_s(J,n)$. In other words
$$
ANC^{wm}(J,n):=NCI_{2,1}(n)\cap NCLP_s^{wm}(J,n)\,.
$$
In the case of monotone labelled partitions, one takes
$$
ANC^{m}(J,n):=NCI_{2,1}(n)\cap NCLP_s^{m }(J,n)\,.
$$
The so-called \emph{nesting forest} representation for noncrossing partitions (see, {\it e.g.} \cite{AHLV}, Definition 3.1) appears useful to describe the latter classes (see Figure \ref{fignesting}). Each irreducible factor is seen as a Hasse diagram, where the labeling is not decreasing moving downward, and the same colour of vertices necessarily corresponds to the same labeling.

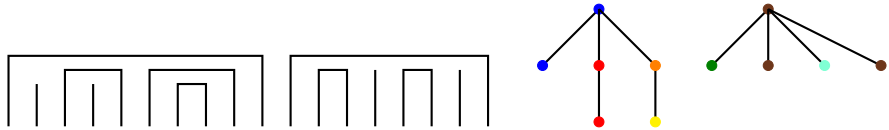
\begin{figure}[h]
\centering
\begin{tikzpicture}[scale=0.75]
     \draw[thick] (0,0) -- +(0,1.25) -| (4.5,0);
     \draw[thick] (0.5,0) -- +(0,0.75);
     \draw[thick] (1,0) -- +(0,1) -| (2,0);
     \draw[thick] (1.5,0) -- +(0,0.75);
     \draw[thick] (2.5,0) -- +(0,1) -| (4,0);
     \draw[thick] (3,0) -- +(0,0.75) -| (3.5,0);
     \draw[thick] (5,0) -- +(0,1.25) -| (8.5,0);
     \draw[thick] (5.5,0) -- +(0,1) -| (6,0);
     \draw[thick] (6.5,0) -- +(0,1);
     \draw[thick] (7,0) -- +(0,1) -| (7.5,0);
     \draw[thick] (8,0) -- +(0,1);
   \end{tikzpicture}
   \hspace{0.5cm}%
   \begin{tikzpicture}[scale=0.75]
   \draw[thick]
(0,0) -- (-1,-1);
\filldraw[blue] (0,0) circle (2.5pt);
\filldraw[blue] (-1,-1) circle (2.5pt);
\draw[thick]
(0,0) -- (0,-1) -- (0,-2);
\filldraw[red] (0,-1) circle (2.5pt);
\filldraw[red] (0,-2) circle (2.5pt);
\draw[thick]
(0,0) -- (1,-1) -- (1,-2);
\filldraw[orange] (1,-1) circle (2.5pt);
\filldraw[yellow] (1,-2) circle (2.5pt);
\draw[thick]
(3,0) -- (2,-1);
\filldraw[auburn] (3,0) circle (2.5pt);
\filldraw[ao(english)] (2,-1) circle (2.5pt);
\draw[thick]
(3,0) -- (3,-1);
\filldraw[auburn] (3,-1) circle (2.5pt);
\draw[thick]
(3,0) -- (4,-1);
\filldraw[aquamarine] (4,-1) circle (2.5pt);
\draw[thick]
(3,0) -- (5,-1);
\filldraw[auburn] (5,-1) circle (2.5pt);
\end{tikzpicture}
\caption{A partition in $ANC^{wm}([7],18)$ and its nesting forest.}
\label{fignesting}
\end{figure}

\subsection{Motzkin and Riordan paths}
\label{moz}
Recall that a \emph{Motzkin path of size $n$} is a lattice path in the integer plane $\bz\times\bz$ consisting of up-steps $a:=(1, 1)$, down-steps $b:=(1,-1)$ and level-steps $c:=(1,0)$, beginning in $(0,0)$ and ending in $(n,0)$, which never passes below the $x$-axis. We denote by $\cam(n)$ the set of all Motzkin paths of size $n$, and $M_n:=|\cam(n)|$ is called $n$th \emph{Motzkin number}.

Following \cite{F1}, $\cam(n)$ reduces to the collection of words $u_1u_2\cdots u_n$ on the alphabet $\{a,b,c\}$ satisfying
\begin{align}
\begin{split}
\label{fj1}
&|u_1u_2\cdots u_j|_a\geq |u_1u_2\cdots u_j|_b\,, \quad 1\leq j\leq n-1\,, \\
&|u_1u_2\cdots u_n|_a = |u_1u_2\cdots u_n|_b\,,
\end{split}
\end{align}
where $|x|_i$ denotes the number of occurrences of $i\in\{a,b,c\}$ in $x$.
\begin{rem}
\label{MD}
Removing the horizontal steps a Motzkin path results in a Dyck path, and immediately leads to the following relation \cite{Ber}
$$
M_n=\sum_{k=0}^{[\frac{n}{2}]}\binom{n}{2k}C_k,
$$
where for any $x\in\br$, $[x]$ is the unique integer $m$ such that $m\leq x< m+1$, and $(C_k)$ is the Catalan sequence.
\end{rem}
\emph{Riordan paths} are Motzkin paths without horizontal steps on the $x$-axis. The set of Riordan paths from $(0,0)$ to $(n,0)$ is denoted by $\car(n)$, and $R_n$ is their number.

Motzkin and Riordan numbers are connected by the following formula \cite{Ber}
\begin{equation}
\label{MnRn}
M_n=R_n+R_{n+1}.
\end{equation}
Any of the aforementioned lattice paths is said \emph{irreducible} if it does not touch the $x$-axis except at the beginning and at the end.

In the sequel it will be useful to replace the alphabet $a,b,c$ with $\dag,1,0$, respectively. In this case, conditions \eqref{fj1} then realize any Motzkin path of length $n$ by means of a sequence $\eps=(\eps(1),\ldots, \eps(n))\in\{1,0,\dag\}^n$ such that

\medskip

(1) $|\eps(1),\ldots, \eps(n)|_1=|\eps(1),\ldots, \eps(n)|_\dag$

\medskip

(2) $|\eps(1),\ldots, \eps(j)|_1 \leq |\eps(1),\ldots, \eps(j)|_\dag$\,,\,\,\, for $j=1,\ldots, n-1$\,.

\medskip

In this picture, $\cam(n)$ is denoted by $\{1,0,\dag\}^n_{M}$. When in particular $\eps(j)\in\{1,\dag\}$ for any $j$, by Remark \ref{MD} $\eps$ uniquely corresponds to a Dyck path of length the necessarily even $n$, or equivalently to a unique partition in $NC_2(n)$. If in addition,

\medskip

(3) $\eps(1)=\dag$

\medskip

(4) $|\eps(1),\ldots, \eps(j)|_1=|\eps(1),\ldots, \eps(j)|_\dag\Longrightarrow\eps(j+1)=\dag$\,,\,\,\, $j\geq 2$

\medskip

there are no horizontal steps in the Motzkin path on the $x$-axis. If $\{1,0,\dag\}^n_{+}$ denotes the sequences $\eps\in\{1,0,\dag\}^n$ satisfying (1) - (4), one finds that $\{1,0,\dag\}^n_{+}\equiv\car(n)$.
This gives sense to the notation $\eps\in\car(n)$ often used in the sequel, and therefore $R_n=|\{1,0,\dag\}^n_{+}|$.

Finally, any $\eps\in\{1,0,\dag\}^n_{+}$ is decomposed as $\eps=\eps'\sqcup \eps''$, where $\eps''(j)\in\{1,\dag\}$ and $\eps'(j)=0$ for any $j$.  Thus $\eps''$ is identified to a Dyck path or, equivalently, $\eps''\in NC_2(2r)$ for some $r\leq\frac{n}{2}$. Here, each block reduces to a pair $(i,j)$ such that $\eps(i)=1$, and $\eps(j)=\dag$. The reader is referred to \cite{CrLu} for similar arguments in the setting of interacting Fock spaces.

\subsection{Cauchy transform of a measure}
\label{sec2c}

Let $\m$ be a probability measure defined on the Borel $\s$-field over $\br$ with finite moments. The moment sequence associated with $\mu$ is denoted by $(m_n(\mu))$. For each $z\in\bc$
$$
\cam_\m(z):=\sum_{n=0}^{\infty}z^nm_n(\mu)
$$
is called moment generating function, which is considered as a formal power series if the series is not absolutely convergent.

From now on $\mathbb{C}^+$ and $\mathbb{C}^-$ will be the upper and lower complex half-planes, respectively.  The Cauchy transform of $\m$ is defined as
\begin{equation*}
\cg_{\mu}(z):=\int_{-\infty}^{+\infty}\frac{\mu(dx)}{z-x}.
\end{equation*}

$\cg_{\mu}(z)$ is analytic in $\bc\setminus \supp({\mu})$, $\cg_{\mu}(\overline{z})=\overline{\cg_{\mu}(z)}$ and $\cg_{\mu}$ maps $\bc^+$ into $\bc^-$. As a consequence, it suffices to consider it on the domain on $\mathbb{C}^+\cup  \mathbb{R}\setminus \supp({\mu})$. In this region it uniquely determines $\m$ as summarized below (see {\it e.g.} \cite{horaob}):

\noindent (1) The limit
\begin{equation}
\label{stiedef}
\gpg(x):=-\frac{1}{\pi}\lim_{y\to 0^{+}}\Im\cg_{\mu}(x+iy)
\end{equation}
exists for a.e. $x\in\br$, and $\gpg(x)\,dx$ is the absolutely continuous part of $\mu$. Formula \eqref{stiedef} is called Stieltjes inversion formula.

\noindent (2) $\cg_{\mu}(z)$ has a simple pole in $z=a\in\mathbb{R}$ if and only if $a$ is an isolated point of $\supp({\mu})$. In this case
$$
\m=c\d_a+(1-c)\nu,\,\,\,\,\, 0\leq c \leq 1
$$
and $\nu$ is a probability measure for which $\supp({\nu})\cap\{a\}=\emptyset$. Furthermore, $c=\text{Res}_{z=a}\cg_{\mu}(z)$.

We end the subsection by noticing that
\begin{equation}
\label{AA}
\cg_\m(z)=\frac{1}{z}\cam_\m\bigg(\frac{1}{z}\bigg)\,, \quad z\in\bc^+\,.
\end{equation}

\subsection{Free Meixner laws}
\label{sec2d}
We end the section briefly presenting free Meixner laws. The interested reader is referred to \cite{KWW,SaYo} for further details.
The family $\mu_{\bfu}$ on $\mathbb{R}$ of free Meixner laws is parametrized by $\bfu:=(a, b, c,\a)\in \mathbb{R}^4$ with $b, c\geq0$, which orthogonalize the family of polynomials $P_n(x):=P^{\bfu}_n(x)$ given by the following recurrence
\begin{align*}
\begin{split}
&P_0(x) := c,\quad P_1(x):=x-\a, \\
&P_{n+1}(x) = (x- a)P_n(x) -bP_{n-1}(x),\quad\text{for $n\geq 1$}\,.
\end{split}
\end{align*}
Moreover, the measure $\mu_{\bfu}$ has absolutely continuous part $\mu_C$ with density
\begin{equation}
\label{ac}
\mu_C(dx)= \frac{c\sqrt{4b-(x-a)^2}}{2\pi f(x)}\chi_{[a-2\sqrt{b}, a +2\sqrt{b}]}(x)\,,
\end{equation}
where
$$
f(x)=(1-c)(x-a)^2+(c-2)(\a-a)(x-a)+(\a-a)^2+bc^2\,.
$$
Its discrete part $\mu_D$ is $0$ except possibly in the following cases:

\bigskip

(1) $f(x)$ has two real roots $x_1\neq x_2$. Then
  \begin{equation}
  \label{muD}
  \mu_D=\g_1\d_{x_1}+\g_2\d_{x_2}\,,
  \end{equation}
  where
  \begin{equation}
  \label{gammai}
  \g_i=\frac{1}{\sqrt{(\a-a)^2-4b(1-c)}}\bigg(\frac{bc}{|x_i-\a|}-\frac{|x_i-\a|}{c}\bigg)_{+}, \quad i=1,2\,.
  \end{equation}

\bigskip

(2) $c=1$ and $\a\neq a$, so that $f(x)$ has one root $x_1=\a+b/(\a-a)$. Then
\begin{equation}
\label{dis}
\mu_D=\bigg(1-\frac{b}{(\a-a)^2}\bigg)_{+}\d_{x_1}\,,
\end{equation}
where $(r)_{+}:=(r+|r|)/2$.

\bigskip

Finally, we recall that the Cauchy transform of $\mu_{\bfu}$  is given by
\begin{equation}
\label{cauchysayo}
\cg_{\mu_{\bfu}}(z)=\frac{\{(c-2)z+(2\a-ac)\}\pm c\sqrt{(z-a)^2-4b}}{-2f(z)}\,,
\end{equation}
where the branch of the analytic square root is determined by the condition $\Im \cg_{\mu_{\bfu}}(z)< 0$ when $\Im z>0$.

\section{the vacuum law of nonsymmetric position operators}
\label{moments}

Motivated by the symmetric central limit theorem for weakly monotone selfadjoint operators \cite{CGW}, we first study the vacuum distribution, \emph{i.e.} the law in the vacuum state $\om_\Om:=\langle\cdot\Om,\Om\rangle$, for sums of nonsymmetric position operators
$$
S_m(\l):=\sum_{k=1}^m \bigg(\frac{G_k}{\sqrt{m}}+\l A^0_k\bigg)\,,
$$
where $\l>0$, $m\geq1$, and $G_k:=A_k+A^{\dag}_k$. The symmetric part, given by the sum of the so-called gaussian operators $G_k$, and obtained by taking $\l=0$, has been already treated in \cite{CGW}. As $S_m(\l)$ is bounded and selfadjoint for each $m$, the moment problem of its vacuum distribution has a unique solution.

For $m\geq 1$, it seems convenient to introduce the operator
$$
T_m(\l_m):=\sqrt{m}S_m(\l)=\sum_{k=1}^m \bigg(G_k+\l_m A^0_k\bigg)\,,
$$
where $\l_m:=\l\sqrt{m}$. Thus, we have to handle the $n$th moments for such sums, namely
$$
b_{m,n}(\l_m):=\om_{\Om}\big(\big(T_m(\l_m)\big)^{n}\big), \quad \l_m>0.
$$
We start with the case $m=1$, where we simply denote the aforementioned objects by $T_1(\l)$ and $b_{1,n}(\l)$.

The law of the nonsymmetric position operator belongs to the free Meixner class, as shown by the following
\begin{thm}
\label{measacRio}
The distribution $\m_{1,\l}$ of $T_1(\l)$ in the vacuum state is the free Meixner law with parameters $(\l,1,1,0)$. Namely,
\begin{equation*}
\m_{1,\l}=
\begin{cases}
\big(1-\frac{1}{\l^2}\big)\delta_{-\frac{1}{\l}}+{\m_{1,\l}}^{a.c.} & \text{if $\l >1$}\\
{\m_{1,\l}}^{a.c.}  & \text{if $0<\l\leq 1$}\,,\\
\end{cases}
\end{equation*}
where ${\m_{1,\l}}^{a.c.}$ is the absolutely continuous part, with the density function
\begin{equation}
\label{de1}
\m^{a.c.}_{1,\l}(dx)=\frac{\sqrt{4-(x-\l)^2}}{2\pi(\l x+1)}\chi_{[\l-2,\l+2]}(x)\,,
\end{equation}
and the above support reduces to $(-1,3]$ when $\l=1$. In addition, for any $n$
\begin{equation}
\label{MOM1}
b_{1,n}(\l)=\sum_{\eps\in\car(n)}\l^{|\eps|_0}\,,
\end{equation}
where $|\eps|_0$ denotes the number of level-steps in $\eps$.
\end{thm}
In Figure \ref{fig:dens_m1} we report the plots of $\m_{1,\l}$ in the cases $\l=0.5$, $\l=1$ and $\l=2$.
\begin{figure}[b]
\centering
\subfloat[][\emph{Distribution for $m=1$ and $\l=0.5$}.]
{\includegraphics[width=.45\textwidth]{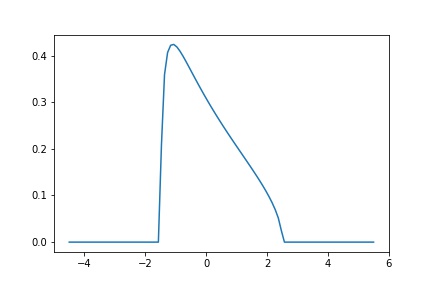}} \quad
\subfloat[][\emph{Distribution for $m=1$ and $\l=1$}.]
{\includegraphics[width=.45\textwidth]{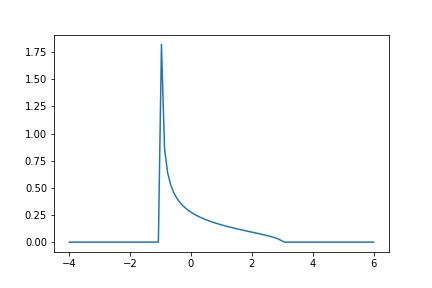}} \\
\subfloat[][\emph{Distribution for for $m=1$ and $\l=2$}.]
{\includegraphics[width=.45\textwidth]{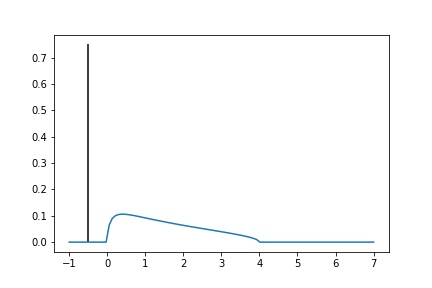}} \quad
\caption{Vacuum distribution of $T_1(\l)$ for $\l=0.5$, $\l=1$ and $\l=2$.}
\label{fig:dens_m1}
\end{figure}
\begin{proof}
We first compute the moment generating function and the Cauchy transform for the sequence $(b_{1,n}(\l))$ which indicate the relation of the latter with Riordan paths. Fix $\l>0$. If for any $i$
$$
\,\,\,\,\,\,\,\,\,\,\,\,\,\,\,\,\,\,\,\,\,\,\,\,\,\,\,\,\,\,\,\,\,\,A_ie_i^{\otimes n}=\left\{
                     \begin{array}{cc}
                       0 & \text{if $n=0$\,,} \\
                       e_i^{\otimes(n-1)} & \text{if $n>0$\,,} \\
                     \end{array}
                   \right.
$$
$$
A^\dag_ie_i^{\otimes n}=e_i^{\otimes(n+1)}\,,
$$
where $n=0$ corresponds to $\Omega$,
$T_1(\l)$ has the following matrix representation
$$
\left(
  \begin{array}{ccccccc}
    0 & 1 & 0 & 0 & 0 & 0 & \cdots  \\
    1 & \l & 1 & 0 & 0 & 0 &  \cdots  \\
    0 & 1 & \l & 1 & 0 & 0 & \cdots  \\
    0 & 0 & 1 & \l & 1 & 0 & \cdots   \\
     &  &  & \ddots & \ddots & \ddots    \\
  \end{array}
\right)
$$
w.r.t. the canonical basis $(e_i^{\otimes n})_{n\geq 0}$. This is a Jacobi matrix, and its $J$-transform is the continued fraction of the Cauchy transform for $(b_{1,n}(\l))$ \cite{Ak}. It is given by
\begin{equation*}
\cg_{1,\l}(z)=\frac{1}{z-\frac{1}{z-\l-\frac{1}{z-\l-\frac{1}{\ddots}}}}\,,
\end{equation*}
and then
\begin{equation}
\label{cauchyR}
\cg_{1,\l}(z)=\frac{z+\l \mp \sqrt{(z-\l)^2-4}}{2(\l z+1)}\,.
\end{equation}
By \eqref{AA} one gets the moment generating function
\begin{equation*}
\cam_{1,\l}(z)=\frac{(1+\l z)-\sqrt{(1-\l z)^2-4z^2}}{2(\l z+z^2)}\,.
\end{equation*}
The continued fraction version assumes the form
\begin{equation}
\label{GENF}
\cam_{1,\l}(z)=\frac{1}{1-\l z-\frac{z^2}{1-\l z-\frac{z^2}{\ddots}}}\,.
\end{equation}
Arguing as in \cite{F1}, Theorem 1, one finds that \eqref{GENF} is indeed the moment generating function of positive lattice paths with weight 1 on up and down-steps, and weight $\l$ on horizontal steps except 0 on the $x$-axis. Thus \eqref{MOM1} is achieved.\\
We note that the Cauchy transform in \eqref{cauchyR} is exactly the one given in \eqref{cauchysayo} for $(a,b,c,\a)=(\l,1,1,0)$, and $f(z)=\l z+1$.
Since
$$
\lim_{x\rightarrow \l+2^-}\frac{\sqrt{4-(x-\l)^2}}{2\pi(\l x+1)}\chi_{[\l-2,\l+2]}(x)=0
$$
and
$$
\lim_{x\rightarrow \l-2^+}\frac{\sqrt{4-(x-\l)^2}}{2\pi(\l x+1)}\chi_{[\l-2,\l+2]}(x)=
\begin{cases}
0 & \text{if $\l\neq 1$}  \\
+\infty & \text{if $\l=1$}\,,
\end{cases}
$$
from \eqref{ac} it follows that the probability measure $\m_{1,\l}$ has the absolutely continuous part given in \eqref{de1}.
Finally, $x_0=-\frac{1}{\l}$ is a root of $f(x)$, and \eqref{dis} entails $x_0$ is an atom with mass
\begin{equation*}
\begin{split}
\bigg(1-\frac{1}{\l^2}\bigg)_{+} = \begin{cases}
0 & \text{if $0<\l\leq 1$}\\
1-\frac{1}{\l^2} & \text{if $\l>1$}.
\end{cases}
\end{split}
\end{equation*}
\end{proof}
One notices that for any $n\geq2$ and $\l>0$, $b_{1,n}(\l)$ is a monic polynomial of degree $n-2$. Indeed, the degree is obtained
by taking the unique path starting with an up-set, followed by $n-2$ level-steps and ending with a down-step. As an example, a direct
computation gives for $n=0,\ldots,5$ the following values for the moment sequence
$$
b_{1,n}(\l)=\{1,0,1,\l,\l^2+2,\l^3+5\l\},\quad \l>0\,.
$$
A recursive formula for these moments is provided by the following
\begin{prop}
\label{recbmlam}
For each $n\geq2$ and $\l>0$, one has
\begin{equation}
\label{bmlMotz}
b_{1,n}(\l)=\sum_{k=2}^{n}b_{1,n-k}(\l) M_{k-2}(\l),
\end{equation}
where
\begin{equation}
\label{Motzlam}
M_m(\l):=\sum_{\eps\in\cam(m)}\l^{|\eps|_0}, \quad M_0(\l)=1.
\end{equation}
\end{prop}
\begin{proof}
Indeed, as $b_{1,0}(\l)=1$, the basic step $n=2$ in \eqref{bmlMotz} is trivially satisfied.

Now we suppose that \eqref{bmlMotz} holds for any $s<n$.
The definition of $\{1,0,\dag\}^n_+\equiv \car(n)$ gives that any Riordan path of length $n$ is uniquely decomposed into an irreducible Riordan path of length $k$, and a Riordan path of length $n-k$, for $k=2,\ldots,n$. If $\overline{\car}(k)$ denotes the set of irreducible Riordan paths of length $k$, from \eqref{MOM1} it follows
\begin{align*}
b_{1,n}(\l)&=\sum_{k=2}^n\sum_{\eps'\in \overline{\car}(k)}\l^{|\eps'|_0} \sum_{\eps''\in \car(n-k)}\l^{|\eps''|_0} \\
&=\sum_{k=2}^n b_{1,n-k}(\l)\sum_{\eps'\in \overline{\car}(k)}\l^{|\eps'|_0}\,,
\end{align*}
the last equality coming from the induction assumption. For any fixed length $k\geq2$, irreducible Riordan paths and irreducible Motzkin paths coincide. In addition, after removing the first and the last step, the latter are in one-to-one correspondence with the  Motzkin paths of length $k-2$ (see Figure \ref{figb1n}). This ends the proof.
\end{proof}

\begin{figure}[h]
\centering
\begin{tikzpicture}[scale=0.50]
\filldraw[thick]
(0,0) circle (2pt) -- (1,1) circle (2pt);
\filldraw[dotted]
(1,1) circle (2pt) -- node[align=center,above] {Irreducible \\ $R_k$} (4,1) circle (2pt);
\filldraw[thick]
(4,1) -- (5,0) circle (2pt);
\filldraw[dotted]
(5,0) circle (2pt) -- (8,0) circle (2pt)
node[pos=0, below]{$k$} node[pos=0.5, below] {$R_{n-k}$} ;
\draw[-Triangle, very thick](9, 1) -- (11, 1);
\filldraw[thick]
(12,0) circle(2pt) -- (13,1) circle (2pt)
node[pos=0.5]{\textcolor[rgb]{1.00,0.00,0.00}{\textbf{$\backslash$}}};
\filldraw[dotted]
(13,1) circle (2pt) -- node[align=center, above] {$M_{k-2}$} (16,1) circle (2pt);
\filldraw[thick]
(16,1) circle (2pt) -- (17,0) circle (2pt)
node[pos=0.5]{\textcolor[rgb]{1.00,0.00,0.00}{\textbf{$/$}}};
\filldraw[dotted]
(17,0) circle (2pt) -- (20,0) circle (2pt)
node[pos=0, below]{$k$} node[pos=0.5, below] {$R_{n-k}$};
\end{tikzpicture}
\caption{}
\label{figb1n}
\end{figure}

As for \eqref{MnRn}, it cannot be generalized to $\l\neq 1$ by means of \eqref{Motzlam}.
Indeed, the sequences of paths considered here are usually called $\l$-colored Motzkin words in combinatorics \cite{ST}.

Although the following result could be known to the experts, we add its proof since some of the tools used will be helpful in the sequel.
\begin{prop}
\label{recMnlam}
For each $n\geq2$ and $\l>0$
\begin{equation*}
M_n(\l)=\l^n+\sum_{k=0}^{n-2}\l^k\sum_{l=k+2}^n M_{n-l}(\l)M_{l-k-2}(\l)\,,
\end{equation*}
where $M_0(\l)=1, M_1(\l)=0$.
\end{prop}
\begin{proof}
Fix a Motzkin path of length $n\geq 2$. Then, it is uniquely decomposed into:

\medskip

\noindent - a path of $k$ level-steps, for $k=0,\ldots, n$;

\medskip

\noindent - an irreducible Motzkin path of length $l-k$, for $l=k+2,\ldots,n$ in the case $k\leq n-2$ (since the $k+1$th step is an up-step);

\medskip

\noindent - a Motzkin path of length $n-l$, in the case $l\leq n-1$.

\medskip

As above, each irreducible Motzkin path of length $l-k$ consists of an up step, a Motzkin path of length $l-k-2$, and a down step (see Figure \ref{figMn}). Therefore,
\begin{align*}
M_n(\l)&=\l^n+\sum_{k=0}^{n-2}\l^k\sum_{l=k+2}^n\sum_{\eps'\in\cam(l-k-2)}\l^{|\eps'|_0}\sum_{\eps''\in\cam(n-l)}\l^{|\eps''|_0} \\
&=\l^n+\sum_{k=0}^{n-2}\l^k\sum_{l=k+2}^n M_{l-k-2}(\l) M_{n-l}(\l)\,.
\end{align*}
\begin{figure}[h]
\centering
\begin{tikzpicture}[scale=0.40]
\filldraw[thick]
(0,0) circle (2pt) -- (1,0) circle (2pt);
\filldraw[dotted]
(1,0) circle (2pt) -- node[align=center, below] {$k$ level \\steps}  (5,0) circle (2pt);
\filldraw[thick]
(5,0) circle (2pt) --(6,0) circle (2pt) node[below]{$k$};
\filldraw[thick]
(6,0) -- (7,1) circle (2pt);
\filldraw[dotted]
(7,1) circle (2pt) -- node[align=center, above] {Irreducible \\$M_{l-k}$} (10,1) circle (2pt);
\filldraw[thick]
(10,1) -- (11,0) circle (2pt);
\filldraw[dotted]
(11,0) circle (2pt) -- (14,0) circle (2pt)
node[pos=0, below]{$l$} node[pos=0.5, below] {$M_{n-l}$} ;
\draw[-Triangle, very thick](15, 1) -- (17, 1);
\filldraw[thick]
(18,0) circle (2pt) -- (19,0) circle (2pt);
\filldraw[dotted]
(19,0) circle (2pt) -- node[align=center, below] {$k$ level\\ steps}  (23,0) circle (2pt);
\filldraw[thick]
(23,0) circle (2pt) --(24,0) circle (2pt) node[below]{$k$};
\filldraw[thick]
(24,0) -- (25,1) circle (2pt)
node[pos=0.5]{\textcolor[rgb]{1.00,0.00,0.00}{\textbf{$\setminus$}}};
\filldraw[dotted]
(25,1) circle (2pt) -- node[ above] {$M_{l-k-2}$} (28,1) circle (2pt);
\filldraw[thick]
(28,1) -- (29,0) circle (2pt)
node[pos=0.5]{\textcolor[rgb]{1.00,0.00,0.00}{\textbf{$/$}}};
\filldraw[dotted]
(29,0) circle (2pt) -- (32,0) circle (2pt)
node[pos=0, below]{$l$} node[pos=0.5, below] {$M_{n-l}$} ;
\end{tikzpicture}
\caption{}
\label{figMn}
\end{figure}
\end{proof}
The Jacobi continued fraction \footnote{which happens to be periodic in this particular example.} and the bivariate generating function for the Motzkin numbers $(M_n(\l))$ is shown in \cite[p. 135]{F1} (with $z$ instead of $\l$):
$$
\cam^M_{\l}(z)= \frac{1}{1-\l z-\frac{z^2}{1-\l z-\frac{z^2}{\ddots}}}\,,
$$
or, equivalently
\begin{equation*}
\cam^M_{\l}(z)=\frac{1}{2z^2}\bigg(1-\l z-\sqrt{(1-\l z)^2-4z^2}\bigg)\,.
\end{equation*}
For the Cauchy transform, \eqref{AA} entails
\begin{equation}
\label{cauchyM}
\cg^M_{\l}(z)=\frac{z-\l\pm \sqrt{(z-\l)^2-4}}{2}\,.
\end{equation}
Here, $''+''$ has to be taken when $\text{Re}(z)\geq \l$, or $\text{Re}(z)<\l$ and $\text{Im}(z)<\text{Im}\sqrt{(z-\l)^2-4}$, whereas one takes $''-''$
in the remaining cases. Recalling that the Cauchy transform of the Wigner law is indeed
$$
\cg^M_{0}(z)=\frac{z\pm \sqrt{z^2-4}}{2}\,,
$$
\eqref{cauchyM} gives that the distribution measure $\m^M_{\l}$ induced by $(M_n(\l))$ results to be a shifted semicircle law. Namely,
\begin{equation*}
\m^M_{\l}(dx)=\frac{1}{2\pi}\sqrt{4-(x-\l)^2}\chi_{[\l-2,\l+2]}(x)\,.
\end{equation*}
In addition, the periodic continued fraction expansion of $\cg^M_{\l}(z)$  is the $J$-transform of the Jacobi matrix \cite{Ak}
$$
\left(
  \begin{array}{ccccccc}
    \l & 1 & 0 & 0 & 0 & 0 & \cdots  \\
    1 & \l & 1 & 0 & 0 & 0 &  \cdots  \\
    0 & 1 & \l & 1 & 0 & 0 & \cdots  \\
    0 & 0 & 1 & \l & 1 & 0 & \cdots   \\
     &  &  & \ddots & \ddots & \ddots    \\
  \end{array}
\right)
$$
which is the matrix representation of every $A_i+A^\dag_i+\l I$ w.r.t. the canonical basis $(e_i^{\otimes n})_{n\geq 0}$, where, as above, $e_i^{\otimes 0}=\Omega$ and $I$ denotes the identity of $\cb(\cf_{WM}(\ch))$. Therefore, $(M_n(\l))$ are the vacuum moments of $A_i+A^\dag_i+\l I$, $i\in \bn$.

\subsection{Distribution of the sum of an arbitrary number of operators}
\label{sec3.1}
Here we investigate the vacuum law for sums of at least 2 nonsymmetric position operators, and show that the moments depend on the number of some partitions introduced in Section \ref{prel}.

Before proving the announced result, we introduce the following technical
\begin{lem}
\label{lemikil}
Fix $n\geq1$, $\eps\in\{1,0,\dag\}^n_{+}$, $i_1,\ldots,i_n\in [m]$, and define
\begin{align}
\begin{split}
\label{twosf}
k &:=\min\{j\leq n\mid \eps(j)=1\}\\
l &:=\max\{i<k\mid \eps(i)=\dag\}\,.
\end{split}
\end{align}
Then, $A_{i_n}^{\eps(n)}\cdots A_{i_1}^{\eps(1)}$ vanishes unless $i_l=i_j=i_k$ for each $j\in[n]$ such that $l\leq j<k$.
\end{lem}
\begin{proof}
Indeed, when $l=k-1$, the thesis follows from \eqref{cr}, since here $j=l$. If $l<k-1$, then for any $j$ such that $l<j<k$, one finds $\eps(j)=0$. In this case \eqref{crcons}, \eqref{cr} and \eqref{jgeqk} give the desired result.
\end{proof}
Notice that for any $m\geq1$, $b_{m,0}(\l_m)=1$ and $b_{m,1}(\l_m)=0$.
\begin{thm}
\label{mommgeq2}
For each $m\geq1$ and $n\geq2$, one has
\begin{equation}
\label{bm4}
b_{m,n}(\l_m)=\sum_{\eps\in ANC^{wm}([m],n)}\l_m^{|\eps|_0},
\end{equation}
where $\l>0$, and $|\eps|_0$ denotes the number of singletons of the partition $\eps$.
In addition,
\begin{equation}
\label{momm}
b_{m,n}(\l_m)=\sum_{k=2}^n b_{m,n-k}(\l_m)\sum_{l=1}^m\bigg( b_{m-l+1,k-2}(\l_m)+ \a_{k,3}\sum_{j=1}^{k-2}\widetilde{b}_{m-l+1,k-2-j}(\l_m)\bigg)\,,
\end{equation}
where, for any $j=1,\ldots, k-2$, $\widetilde{b}_{m-l+1,k-2-j}(\l_m)$ is defined as
$$
b_{m-l+1,k-2-j}(\l_m)\bigg(\a_{[\frac{k-2-j}{2}],1}\sum_{r=1}^{[\frac{k-2-j}{2}]}\binom{r+j}{j}\l_m^j
+\d_{[\frac{k-2-j}{2}],0}\l_m^j\bigg)\,.
$$
\end{thm}
\begin{proof}
We first note that for each $m\geq2$, $n\geq0$, and $\l>0$
\begin{equation}
\label{bmn2}
b_{m,n}(\l_m)=\sum_{\eps\in\{1,0,\dag\}^n_{+}}\sum_{i_1,\ldots,i_n=1}^m  \om_{\Om}(A_{i_n}^{\eps(n)}\cdots A_{i_1}^{\eps(1)})\l_m^{|\eps|_0}\,.
\end{equation}
By the definition of $\{1,0,\dag\}^n_+$, any partition $\eps$ on the right hand side of \eqref{bmn2} has no outer singletons. As previously pointed out, $\eps=\eps' \sqcup \eps''$, where $\eps''$ is a noncrossing pair partition, and now $\eps'$ collects the preservation operators organized in singletons. This gives $\eps\in NCI_{2,1}(n)$. Moreover, the Riordan path $\eps$ is uniquely decomposed into irreducible paths, the first one being $\{\dag,\eps(2),\ldots,\eps(j-1),1\}$, for some $j=2,\ldots n$. Lemma \ref{lemikil} and \eqref{jgeqk} give that the sequence
\begin{equation*}
A^{1}_{i_j}A^{\eps(j-1)}_{i_{j-1}}\cdots A^{\eps(2)}_{i_{2}}A^{\dag}_{i_1}
\end{equation*}
reduces to
$$
\d_{i_l,i_k}A^{1}_{i_j}A^{\eps(j-1)}_{i_{j-1}}\cdots A^{1}_{i_k}A^{\dag}_{i_l}\cdots A^{\dag}_{i_1}\,,
$$
for $l$ and $k$ defined as in \eqref{twosf}. By \eqref{alpha}, the above expression turns out to be equal to
\begin{equation}
\label{***}
\d_{i_l,i_k}\a_{i_{l},i_{l-1}}A^{1}_{i_j}A^{\eps(j-1)}_{i_{j-1}}\cdots A^{\eps({k+1})}_{i_{k+1}}A^{\eps({l-1})}_{i_{l-1}}\cdots A^{\dag}_{i_1}\,.
\end{equation}
Now both $\eps(l-1)$ and $\eps(l-2)$ belong to $\{1,\dag\}$, and the following cases may appear
\begin{align*}
\begin{split}
& A^0_{i_{l-1}}A^{\dag}_{i_{l-2}}=\d_{i_{l-1},i_{l-2}}A^{\dag}_{i_{l-2}} \\
& A^{\dag}_{i_{l-1}}A^{\dag}_{i_{l-2}}=\a_{i_{l-1},i_{l-2}}A^{\dag}_{i_{l-1}}A^{\dag}_{i_{l-2}} \\
& A^0_{i_{l-1}}A^{0}_{i_{l-2}}=\d_{i_{l-1},i_{l-2}}A^0_{i_{l-1}} \\
& A^{\dag}_{i_{l-1}}A^{0}_{i_{l-2}}=\a_{i_{l-1},i_{l-2}}A^{\dag}_{i_{l-1}}A^{0}_{i_{l-2}}\,.
\end{split}
\end{align*}
After replacing in \eqref{***} and iterating the procedure, one finds that any irreducible path realizes a partition which is weakly monotone ordered, where any block of cardinality 1 inherits the labeling from the block which covers. Since by Lemma \ref{lem2.1WM} $\om_{\Om}(A_{i_n}^{\eps(n)}\cdots A_{i_1}^{\eps(1)})=1$, \eqref{bm4} follows.

In the aforementioned decomposition for $\eps\in ANC^{wm}([m],n)$, consider the first irreducible path of length $k=2,\ldots, n$, and take the remaining path of length $n-k$. Consequently, if $ANC^{\overline{wm}}([m],k)$ is the subset of irreducible partitions $\pi=\{B_1,\ldots,B_p\}$ in $ANC^{wm}([m],n)$, from \eqref{bm4} one finds
\begin{align*}
b_{m,n}(\l_m)&=\sum_{k=2}^n\sum_{\eps'\in ANC^{\overline{wm}}([m],k)}\l^{|\eps'|_0}_m\sum_{\eps''\in ANC^{wm}([m],n-k)}\l^{|\eps''|_0}_m \\
&=\sum_{k=2}^nb_{m,n-k}(\l_m)\sum_{\eps'\in ANC^{\overline{wm}}([m],k)}\l^{|\eps'|_0}_m\,.
\end{align*}
We label now $L(B_1)=l$, $l=1,\ldots,m$. The weakly monotone ordering gives that for each $h=2,\ldots, p$, $L(B_h)=l+s$ for a suitable $s\in \{0,\ldots, m-l\}$. Removing $B_1$ provides a natural identification with a (unique) partition in $ANC^{wm}([l,m],k-2-j)$, where $j$ is the number of singletons. Consequently, for $j=0$ \eqref{momm} follows from \eqref{bm4}, since in this case
$$
ANC^{\overline{wm}}([m],k)=\bigsqcup_{l=1}^m ANC^{wm}([m-l+1],k-2)\,,
$$
where we identified $[m-l+1]$ with $[l,m]$.
If instead $j\geq 1$, denote by $r$ the number of outer blocks with cardinality 2 in the partition. Then the indistinguishable outer singletons may occupy $r+1$ positions, and \eqref{momm} is verified using again \eqref{bm4}.
\end{proof}
The computation of the moment generating function and the Cauchy transform for $(b_{m,n}(\l_m))$ appears complicated when $m\geq 2$, as a consequence of the recurrence relation \eqref{momm}.

In particular, when we replace $\l_m$ simply with $\l$ for any $m$, the vacuum law of the sum of $m$ nonsymmetric position operators is the $m$-fold monotone convolution of $\m_{1,\l}$. Indeed, by means of Theorem 2.2 in \cite{CGW}, the random variables $G_i+\l A^0_i$, $i=1,\ldots, m$ are monotone independent. Also in this case, the (reciprocal) Cauchy transform appears not to be simple to handle. As an example, in the easiest case $m=2$, since Theorem 3.1 in \cite{Mu}, one finds $\cg_{2,\l}(z)=\cg_{1,\l}(\frac{1}{\cg_{1,\l}(z)})$. This gives
$$
\cg_{2,\l}(z)=\frac{1+\l \cg_{1,\l}(z)}{2\big(\l+\cg_{1,\l}(z)\big)}\mp\frac{\sqrt{\bigg(\frac{1}{\cg_{1,\l}(z)}-\l\bigg)^2-4}}{2\bigg(\frac{\l}{\cg_{1,\l}(z)}+1\bigg)}\,, \quad z\in\bc^+
$$
as a consequence of  \eqref{cauchyR}.

\section{central limit theorem}
\label{sec5}
In this section we provide a central limit theorem for the operator $T_m(\l_m)$. More in detail, we find that the limit law belongs to the free Meixner class.
\begin{thm}
\label{clt}
For any $\l>0$, the vacuum law of
$$\frac{T_m(\l_m)}{\sqrt{m}}=\frac{G_1+\ldots + G_m+\l_m(A^0_1+\ldots +A^0_m)}{\sqrt{m}}
$$
weakly converges for $m\rightarrow \infty$ to the free Meixner distribution $\nu_\l$ with parameters $\displaystyle\bigg(\l,\frac{1}{2},2,0\bigg)$. Namely,
\begin{equation}
\label{cc}
\nu_\l=\frac{\l}{\sqrt{\l^2+2}}\d_{\l-\sqrt{\l^2+2}}+\nu_\l^{\text{a.c.}}\,,
\end{equation}
where the absolutely continuous part w.r.t. the Lebesgue measure $\nu_\l^{\text{a.c.}}$ has the density
\begin{equation}
\label{ccc}
\nu_\l^{\text{a.c.}}(dx)=\frac{\sqrt{2-(x-\l)^2}}{\pi(\l^2+2-(x-\l)^2)}\chi_{(\l-\sqrt2,\l+\sqrt2)}(x)\,.
\end{equation}
\end{thm}
\begin{proof}
As for any $m\geq 1$ $T_m(\l_m)$ is selfadjoint, it is enough to prove that the convergence is in the sense of moments to $\nu_\l$, \emph{i.e.}
$$
\lim_{m\rightarrow \infty} \frac{1}{\sqrt{m^n}}b_{m,n}(\l_m)=\int_{\br}x^n \di{\nu_\l}(x),
$$
for any $n\geq 2$. From Theorem \ref{mommgeq2} it follows
\begin{equation}
\label{cl1}
\frac{1}{\sqrt{m^n}}b_{m,n}(\l_m)=\frac{1}{\sqrt{m^n}}\sum_{\eps\in ANC^{wm}([m],n)}\l_m^{|\eps|_0}\,.
\end{equation}
As usual, any path in $ANC^{wm}([m],n)$ is uniquely decomposed into $j=1,\ldots, [\frac{n}{2}]$ irreducible paths of size $q_h$, $h=1,\ldots, j$. The latter indeed realize a partition $\pi^{(j)}\in V(n)$ whose corresponding composition of $n$ is $(q_1,\ldots,q_j)$. Therefore, \eqref{cl1} gives
\begin{equation}
\label{cl2}
\frac{1}{\sqrt{m^n}}b_{m,n}(\l_m)=\frac{1}{\sqrt{m^n}}\sum_{j=1}^{[\frac{n}{2}]}\,\,\,\sum_{\pi^{(j)}\in V(n)}\,\,\,\prod_{h=1}^j\,\,\,\sum_{\eps_h\in ANC^{wm}([m],q_h)}\l_m^{|\eps_h|_0}\,.
\end{equation}
For any $h$, let $k_h+1$ be the number of blocks of cardinality 2 of $\eps_h$. Thus, the number of singleton blocks is $q_h-2-2k_h$, and they are labelled as any block they cover. The right hand side of \eqref{cl2} is then
$$
\frac{1}{\sqrt{m^{n}}}\sum_{j=1}^{[\frac{n}{2}]}\sum_{\pi^{(j)}\in V(n)}\prod_{h=1}^j\sum_{k_h=0}^{[\frac{q_h-2}{2}]}\sum_{l=1}^{k_h+1}\binom{m}{l}\binom{q_h-2}{2k_h}
\l^{q_h-2-2k_h}m^{\frac{q_h}{2}-1-k_h}|\ca_{q_h,l}|\,.
$$
Here, $l$ is the cardinality of the range of the label function $L_h:\eps_h\rightarrow [m]$, and $\ca_{q_h,l}$ is the set of all the irreducible paths of length $q_h$ with $l$ blocks. Taking the limit for $m\rightarrow\infty$ all the terms except $l=k_h+1$ vanish, and one has
$$
\lim_{m\rightarrow \infty} \frac{1}{\sqrt{m^n}}b_{m,n}(\l_m)=\sum_{j=1}^{[\frac{n}{2}]}\sum_{\pi^{(j)}\in V(n)}\prod_{h=1}^j\sum_{k_h=0}^{[\frac{q_h-2}{2}]}\binom{q_h-2}{2k_h}
\l^{q_h-2-2k_h}\frac{|\ca_{q_h,k_h+1}|}{(k_h+1)!}\,.
$$
Every path above is irreducible. Then, by usual arguments, for the computation of $|\ca_{q_h,k_h+1}|$ we reduce to the case of a partition $\widetilde{\eps}_h$ of $k_h$ blocks of cardinality 2. As the range of the label function $L_{\widetilde{\eps}_h}$ has cardinality $k_h$, without loss of generality we suppose $\text{Range}(L_{\widetilde{\eps}_h})=[k_h]$. Under the above assumptions, there is a unique block $B_{h_v}$ in $\widetilde{\eps}_h$ connecting two index consecutive elements and such that $L_{\widetilde{\eps}_h}(B_{h_v})=k_h$. Thus, $B_{h_v}$ can be chosen in $(2k_h-1)$ ways, and the same argument holds for $\widetilde{\eps}_h\setminus \{B_{h_v}\}$. Consequently, an iteration procedure gives that $|\ca_{q_h,k_h+1}|=(2k_h-1)!!$. Recalling that $A_n:=\frac{(2n-1)!!}{n!}$ is the $2n$th moment of the standard arcsine law \cite{[Lu95a], Mur}, one has
$$
\lim_{m\rightarrow \infty} \frac{1}{\sqrt{m^n}}b_{m,n}(\l_m)=\sum_{j=1}^{[\frac{n}{2}]}\sum_{\pi^{(j)}\in V(n)}\prod_{h=1}^j\sum_{k_h=0}^{[\frac{q_h-2}{2}]}\binom{q_h-2}{2k_h}
\l^{q_h-2-2k_h}\frac{A_{k_h}}{k_h+1}\,.
$$
Therefore, the moment generating function for the measure $\nu$ is
\begin{equation*}
\mathcal{M}_{\nu_{\l}}(z)= 1+\sum_{n=2}^{\infty}\left(\sum_{j=1}^{[\frac{n}{2}]}\sum_{\pi^{(j)}\in V(n)}\prod_{h=1}^j\sum_{k_h=0}^{[\frac{q_h-2}{2}]}\binom{q_h-2}{2k_h}
\l^{q_h-2-2k_h}\frac{A_{k_h}}{k_h+1}\right)z^n
\end{equation*}
in its set of convergence.
After replacing any $\pi^{(j)}\in V(n)$ with its corresponding composition $(q_1,\ldots,q_j)$ of $n$, the identity $z^n=\prod_{h=1}^j z^{q_h}$ allows to reduce the right hand side as follows
$$
1+\sum_{j=1}^{\infty}\left(\sum_{q=2}^{\infty}\sum_{k_1=0}^{[\frac{q-2}{2}]}\binom{q-2}{2k_1}
\l^{q-2-2k_1}\frac{A_{k_1}}{k_1+1}z^{q}\right)^j\,,
$$
and $q$ denotes any of the $q_h$, $h=1,\ldots,j$.
This means that
\begin{equation}
\label{cl3}
\mathcal{M}_{\nu_\l}(z)=\frac{1}{1-\ck_{\l}(z)}\,,
\end{equation}
where
$$
\ck_{\l}(z):=z^2\sum_{r=0}^{\infty}\left(\sum_{k=0}^{[\frac{r}{2}]}\binom{r}{2k}\,
\l^{r-2k}\frac{A_{k}}{k+1}\right)z^r\,.
$$
As $\frac{A_{k}}{k+1}=\frac{C_{k}}{2^k}$, $(C_k)$ being as usual the Catalan sequence, one has
\begin{align*}
\ck_{\l}(z)=&z^2\sum_{k=0}^{\infty}\frac{C_{k}}{(2\l^2)^k}\sum_{r=2k}^{\infty}\binom{r}{2k}(\l z)^r \\
=&\frac{z^2}{1-\l z}\sum_{k=0}^{\infty}C_{k}\bigg(\frac{z}{\sqrt2(1-\l z)}\bigg)^{2k}\,,
\end{align*}
and in the last equality we used the well known identity
$$
\sum_{r=k}^{\infty}\binom{r}{k}z^r=\frac{z^{k}}{(1- z)^{k+1}}
$$
(see, \emph{e.g.} \cite{GKP}).
Recalling that the moment generating function of the standard Wigner law is
$$
\sum_{k=0}^{\infty}C_{k}z^{2k}=\frac{1-\sqrt{1-4z^2}}{2z^2}\,,
$$
(see, \emph{e.g.} \cite{CGW}),
it turns out that
$$
\ck_{\l}(z)=(1-\l z)\mp\sqrt{(1-\l z)^2-2z^2}\,.
$$
From \eqref{cl3}, after taking into account that $\lim_{z\to0} \cf_{\nu_\l}(z)= 1$, one has that the moment generating function is
\begin{equation}
\label{mgfnu}
\cam_{\nu_\l}(z)=\frac{1}{\l z+\sqrt{(1-\l z)^2-2z^2}}\,, \quad z\in\bc\,.
\end{equation}
By \eqref{AA}, \eqref{mgfnu} entails
$$
\cg_{\nu_\l}(z)=\frac{1}{\l\pm \sqrt{(z-\l)^2-2}}\,.
$$
The previous formula represents the Cauchy transform of the free Meixner law with parameters $\displaystyle\bigg(\l,\frac{1}{2},2,0\bigg)$. Thus, $$
f(x)=\l^2+2-(x-\l)^2
$$
has two real roots $x_1=\l-\sqrt{\l^2+2}$, and $x_2=\l+\sqrt{\l^2+2}$. By \eqref{gammai}, one finds $\g_1=0$ and $\g_2=\frac{\l}{\sqrt{\l^2+2}}$.\\
As a consequence, from \eqref{muD} and \eqref{ac} one has \eqref{cc} and \eqref{ccc}.
\end{proof}

In Figure \ref{fig:limitdistribution} we report the plots of $\nu_\l$ for $\l=1$ and $\l=4$.
\begin{figure}[h]
\centering
\subfloat[][\emph{Limit distribution for $\l=1$}.]
{\includegraphics[width=.45\textwidth]{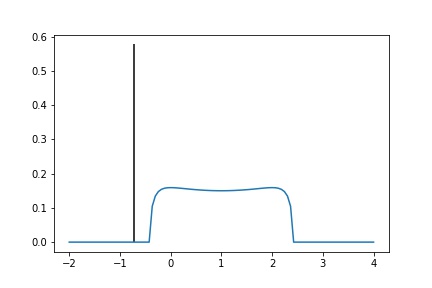}} \quad
\subfloat[][\emph{Limit distribution for $\l=4$}.]
{\includegraphics[width=.45\textwidth]{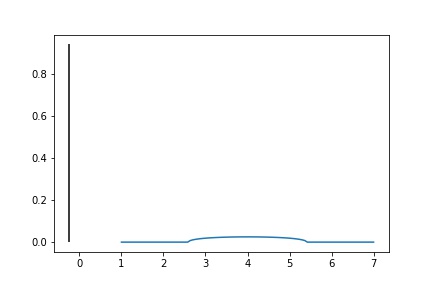}}
\caption{Limit distribution for $\l=1$ and $\l=4$.}
\label{fig:limitdistribution}
\end{figure}

\section{the case of monotone fock space}
\label{secfin}
The previously achieved results can be translated, up to modifications, to the case of monotone Fock space. Some of the following properties have been already presented in \cite{Mur1999}. We add them here for the sake of completeness, and to emphasise their role in our approach.

For our aim it appears at first useful to recall some notions about discrete monotone Fock space, the reader being referred to \cite{CFL1,CFL,Mur} for further details.

Let $\ch$ be a separable Hilbert space, and $(e_i)_{i\geq 1}$ a fixed orthonormal basis. The \emph{monotone Fock space} over $\ch$, in the sequel denoted by $\gf_{M}(\ch)$, is the closed subspace of the full Fock space $\gf(\ch)$ spanned by $\Om$, $\ch$ and all the simple tensors of the form $e_{i_k}\otimes e_{i_{k-1}}\otimes \cdots\otimes e_{i_1}$, where $i_k> i_{k-1}>\cdots> i_1$,
and $k\geq2$.
Let $(i_1,i_2,\ldots,i_k)$ be a strictly decreasing sequence of natural integers. The generic element of the canonical basis of $\gf_M$ is denoted by $e_{(i_1,i_2,\ldots,i_k)}$. Very often, we write $e_{(i)}$ as $e_i$ to simplify the notations.
The monotone creation and annihilation operators with test function $e_i$ for any $i\in \mathbb{N}$, are denoted by $a^\dag_i:=a^\dag(e_i)$ and $a_i:=a(e_i)$, respectively. They are given by $a^\dag_i\Om=e_i$, $a_i\Om=0$ and
\begin{equation*}
a^\dagger_i e_{(i_1,i_2,\ldots,i_k)}:=\left\{
\begin{array}{ll}
e_{(i,i_1,i_2,\ldots,i_k)} & \text{if}\, i> i_1 \\
0 & \text{otherwise}, \\
\end{array}
\right.
\end{equation*}
\begin{equation*}
\,\,\,\,\,\,\,\,\,\,\,\,\,\,\,\,\,\,a_ie_{(i_1,i_2,\ldots,i_k)}:=\left\{
\begin{array}{ll}
e_{(i_2,\ldots,i_k)} & \text{if}\, k\geq 1,\,\,\,\,\,\, \text{and}\,\,\,\,\,\, i=i_1\\
0 & \text{otherwise}. \\
\end{array}
\right.
\end{equation*}
One can check that both $a^\dagger_i$ and $a_i$ have unital norm, are mutually adjoint, and satisfy the following relations
\begin{equation}
\label{comrul}
\begin{array}{ll}
  a^\dagger_ia^\dagger_j=a_ja_i=0 & \text{if}\,\, i\leq j\,, \\
  a_ia^\dagger_j=0 & \text{if}\,\, i\neq j\,.
\end{array}
\end{equation}
As in the weakly monotone case, we introduce the preservation operator $a_i^0:=a_i^\dag a_i$, $i\in\bn$, and recall that Lemma 5.4 in \cite{CFL1} gives
$$
a_ka_ja^\dag_j=\d_{k)}(j)a_k\,, \quad a_ja_ka^\dag_k =\d_{k)}(j)a^\dag_k\,,
$$
where
$$
\d_{k)}(j):=\begin{cases}
1 & \text{if $j<k$} \\
0 & \text{otherwise}\,,
\end{cases}
$$
and further, for $j\leq k$
$$
a_ja^\dag_ja_k=a_k\,, \quad a^\dag_ka_ja^\dag_j=a^\dag_k\,.
$$
We deal with the distribution in the vacuum state of the operator
$$
P_m(\l_m):=\sum_{i=1}^m a_i+a_i^\dag + \l_m a^0_i\,,
$$
where $m\geq 1$, $\l>0$ and $\l_m:=\l\sqrt{m}$. As usual, the law of $P_m(\l_m)$ is determined by the moments, and we introduce the following notation
$$
g_{m,n}(\l_m):=\om_{\Om}((P_m(\l_m))^n)\,.
$$
We start from the case $m=1$.
\begin{prop}
\label{mon1}
For any $\l>0$ the vacuum distribution of $P_1(\l)$ is the two-points law
\begin{equation}
\label{monlaw}
\r=\bigg(\frac{1}{2}-\frac{\l}{2\sqrt{\l^2+4}}\bigg)\d_{\frac{\l+\sqrt{\l^2+4}}{2}} + \bigg(\frac{1}{2}+\frac{\l}{2\sqrt{\l^2+4}}\bigg)\d_{\frac{\l-\sqrt{\l^2+4}}{2}}\,.
\end{equation}
\end{prop}
\begin{proof}
Fix $\l>0$. The definition of monotone creation and annihilation operators gives the following Jacobi matrix representation for $P_1(\l)$
$$
\left(
  \begin{array}{cc}
    0 & 1 \\
    1 & \l \\
  \end{array}
\right)
$$
w.r.t. the canonical basis. Consequently, as in \cite{Ak} we can write down the Cauchy transform $\mathcal{C}_{1,\l}$ of the vacuum law of $P_1(\l)$
$$
\mathcal{C}_{1,\l}(z)=\frac{1}{z-\frac{1}{z-\l}}=\frac{z-\l}{z^2-\l z -1}\,.
$$
Finally, \eqref{monlaw} follows as $\mathcal{C}_{1,\l}(z)$ has two simple poles in $\frac{\l\pm \sqrt{\l^2+4}}{2}$, and
$$
\text{Res}_{z=z_0}\mathcal{C}_{1,\l}(z)=\frac{1}{2}\mp \frac{\l}{2\sqrt{\l^2+4}}\,,
$$
where $z_0:=\frac{\l\pm \sqrt{\l^2+4}}{2}$.
\end{proof}
Notice that for $n\geq 2$, $g_{1,n}(1)$ is the Fibonacci sequence. For the general case covering the circumstance $m\geq 2$, we have the analogue of Theorem \ref{mommgeq2}. As usual, $g_{m,0}(\l_m)=1$, and $g_{m,1}(\l_m)=0$ for any $m\geq 1$.
\begin{prop}
\label{monmom}
For each $m\geq1$, $n\geq2$ and $\l>0$ one has
\begin{equation}
\label{bm11}
g_{m,n}(\l_m)=\sum_{\eps\in ANC^{m}([m],n)}\l_m^{|\eps|_0}\,,
\end{equation}
where $|\eps|_0$ denotes the number of singletons in the partition $\eps$.
In addition, the following recursion formula holds
\begin{equation}
\label{momm11}
g_{m,n}(\l_m)=\sum_{k=2}^n g_{m,n-k}(\l_m)\sum_{l=1}^m \bigg(g_{m-l+1,k-2}(\l_m)+\a_{k,3}\sum_{j=1}^{k-2}\widetilde{g}_{m-l+1,k-2-j}(\l_m)\bigg)\,,
\end{equation}
where for each $j=1,\ldots, k-2$, $\widetilde{g}_{m-l+1,k-2-j}(\l_m)$ is defined as
$$
g_{m-l+1,k-2-j(\l_m)}\bigg(\a_{[\frac{k-2-j}{2}],1}\sum_{r=1}^{[\frac{k-2-j}{2}]}\binom{r+j}{j}\l_m^j + \d_{[\frac{k-2-j}{2}],0}\l_m^j\bigg)\,.
$$
\end{prop}
\begin{proof}
Notice that \eqref{bm11} and \eqref{momm11} directly follow from \eqref{bm4} and \eqref{momm}, respectively. In fact, as a consequence of \eqref{comrul}, for any $n$ the nonvanishing partitions are those $\pi:=\{B_1,\ldots,B_p\}\in ANC^{wm}([m],n)$ for which $L(B_i)<L(B_j)$ when $|B_i|=|B_j|=2$, and $\min B_i<\min B_j<\max B_j<\max B_i$, \emph{i.e.} the partitions belonging to $ANC^{m}([m],n)$.
\end{proof}
In the next lines we show that the rescaled sums of nonsymmetric position operators in the monotone case weakly converge in the vacuum state to the same limit distribution $\nu_\l$ found in the weakly monotone case. This exactly reflects what occurs in the symmetric case, where the standard arcsine law is the central limit distribution for position operators both in monotone \cite{Mur} and in the weakly monotone \cite{CGW} cases.
\begin{thm}
\label{monoclt}
For any $\l>0$, the vacuum law of $\displaystyle \frac{P_m(\l_m)}{\sqrt{m}}$
weakly converges for $m\rightarrow \infty$ to the distribution $\nu_\l$ given in \eqref{cc} and \eqref{ccc}.
\end{thm}
\begin{proof}
Indeed, one obtains the same result of Theorem \ref{clt}. This is achieved using the arguments developed in its proof, and taking into account that:

\medskip

\noindent (1) the labels for singletons in the monotone and weakly monotone cases satisfy the same bounds,

\medskip

\noindent (2) under the notations introduced in the proof of Theorem \ref{clt}, when $l=k_h+1$, for $h=1,\ldots, j$ and $j=1,\dots, [\frac{n}{2}]$, the same partitions are involved for both weakly monotone and monotone label functions.
\end{proof}

\section*{Data availability statement}
Data sharing not applicable to this article as no datasets were generated or analysed during the current study.

\section*{Acknowledgements}
\noindent V. Crismale and M. E. Griseta kindly acknowledge the support of the italian INDAM-GNAMPA and Fondi di Ateneo Universit\`a di Bari ``Probabilit\`a Quantistica e Applicazioni''. V. Crismale also aknowledges  the FFABR project of italian MIUR. J. Wysocza\'nski kindly aknowledges the support of the Polish National Center for Science grant 2016/21/B/ST1/00628.

\end{document}